\documentclass[11pt, a4paper]{article}
\usepackage{latexsym,amssymb}
\usepackage{amsmath}
\usepackage[mathscr]{eucal}
\usepackage{color}
\usepackage[abbrev]{amsrefs}

\newtheorem{Theorem}{\bf Theorem}[section]
\newtheorem{Lemma}{\bf Lemma}[section]
\newtheorem{Proposition}{\bf Proposition}[section]
\newtheorem{Corollary}{\bf Corollary}[section]
\newtheorem{Remark}{\bf Remark}[section]
\newtheorem{Example}{\bf Example}[section]
\newtheorem{Definition}{\bf Definition}[section]

\newenvironment{theorem}{\begin{Theorem}$\!\!\!$}{\end{Theorem}}
\newenvironment{lemma}{\begin{Lemma}$\!\!\!$}{\end{Lemma}}
\newenvironment{proposition}{\begin{Proposition}$\!\!\!$}{\end{Proposition}}
\newenvironment{corollary}{\begin{Corollary}$\!\!\!$}{\end{Corollary}}
\newenvironment{remark}{\begin{Remark}$\!\!\!$}{\end{Remark}}

\newenvironment{definition}{\begin{Definition}$\!\!\!$}{\end{Definition}}

\numberwithin{equation}{section}

\numberwithin{equation}{section}

\newcommand{\dee}{{\rm{d}}}

\def\XXint#1#2#3{{\setbox0=\hbox{$#1{#2#3}{\int}$}
\vcenter{\hbox{$#2#3$}}\kern-.5\wd0}}

\allowdisplaybreaks

\begin{document}
\title{Preservation of $F$-convexity under the heat flow}
\author{
%\qquad\\
Kazuhiro Ishige, Troy Petitt, and Paolo Salani}
\date{}
\maketitle
\begin{abstract}
We introduce the notion of $F$-convexity as a general extension of power convexity.
We characterize the $F$-convexities preserved under the heat flow in the $n$-dimensional Euclidean space, and identify the strongest and the weakest ones among them.
We also characterize the $F$-convexities preserved under the Dirichlet heat flow in convex domains.
\end{abstract}

\vspace{40pt}
\noindent 
Addresses:

\smallskip
\noindent 
\smallskip
\noindent 
K. I.: Graduate School of Mathematical Sciences, The University of Tokyo,\\ 
3-8-1 Komaba, Meguro-ku, Tokyo 153-8914, Japan\\
\noindent 
E-mail: {\tt ishige@ms.u-tokyo.ac.jp}\\

\smallskip
\noindent 
T. P.: Departamento de Matem\'aticas, Universidad Carlos III de Madrid,\\
Avda. de la Universidad, 30. 28911 Legan\'es, Spain\\
\noindent 
E-mail: {\tt tpetitt@math.uc3m.es}\\

\smallskip
\noindent 
P. S.: Dipartimento di Matematica ``U. Dini", Universit\`{a} di Firenze,\\ viale Morgagni 67/A, 50134 Firenze, Italy\\
\noindent 
E-mail: {\tt paolo.salani@unifi.it}\\

\newpage
%%%%%%%%%%%%%%%%%%%%%%
%%%%%%%%%%%%%%%%%%%%%%
\section{Introduction}
%%%%%%%%%%%%%%%%%%%%%%
%%%%%%%%%%%%%%%%%%%%%%
Consider the Cauchy problem for the heat equation, 
\begin{equation}
\tag{H}
\label{eq:H}
\left\{
\begin{array}{ll}
\partial_t u=\Delta u,\quad & x\in{\mathbb R}^n,\,\,\,t>0,\vspace{3pt}\\
u(x,0)=\phi(x),\quad & x\in{\mathbb R}^n, 
\end{array}
\right.
\end{equation}
where $n\ge 1$ and $\phi$ is a nonnegative continuous function on ${\mathbb R}^n$. 
Problem~\eqref{eq:H} admits a nonnegative classical solution in ${\mathbb R}^n\times(0,T)$ for some $T\in(0,\infty]$ 
if and only if 
\begin{equation}
\label{eq:1.1}
\int_{{\mathbb R}^n}e^{-\gamma|x|^2}\phi(x)\,\dee x<\infty\quad\text{for some $\gamma>0$}.
\end{equation}
Furthermore, under condition~\eqref{eq:1.1}, 
the function $e^{t\Delta}\phi$ defined by 
\begin{equation}
\label{eq:1.2}
\left[e^{t\Delta}\phi\right](x):=
\begin{cases}
\displaystyle\int_{{\mathbb R}^n}\Gamma_n(x-y,t)\phi(y)\,\mathrm{d}y, 
  & (x,t)\in{\mathbb R}^n\times(0,\infty),\\[8pt]
\phi(x), & (x,t)\in{\mathbb R}^n\times\{0\},
\end{cases}
\end{equation}
is the unique nonnegative classical solution to problem~\eqref{eq:H} 
(see, e.g., \cite{A}*{Theorem~11}). Here and throughout the paper, $\Gamma_n$ denotes the Gauss kernel on ${\mathbb R}^n\times(0,\infty)$, that is, 
$$
\Gamma_n(x,t):=\,(4\pi t)^{-\frac{n}{2}}\exp\left(-\frac{|x|^2}{4t}\right)
\quad\text{for } (x,t)\in{\mathbb R}^n\times(0,\infty).
$$ 
The representation formula~\eqref{eq:1.2} directly yields the preservation of convexity and log-convexity  {under} the heat flow.
Indeed, denoting by $T(\phi)$ the maximal existence time of $e^{t\Delta}\phi$,
we obtain the following properties.
\begin{itemize}
  \item[(A1)]
  Assume that $\phi$ is convex in ${\mathbb R}^n$. Then 
  \begin{equation*}
  \begin{split}
   & \left[e^{t\Delta}\phi\right]((1-\lambda)x_0+\lambda x_1) =\int_{{\mathbb R}^n}\Gamma_n(y,t)\phi((1-\lambda)x_0+\lambda x_1-y)\,\dee y\\
   & \qquad\quad
   \le\int_{{\mathbb R}^n} {\Gamma_n}(y,t)\left\{(1-\lambda)\phi(x_0-y)+\lambda\phi(x_1-y)\right\}\,\dee y\\
   & \qquad\quad
   =(1-\lambda)\left[e^{t\Delta}\phi\right](x_0)+\lambda \left[e^{t\Delta}\phi\right](x_1)
  \end{split}
  \end{equation*}
  for $x_0$, $x_1\in{\mathbb R}^n$, $t\in (0,T(\phi))$, and $\lambda\in(0,1)$. 
  Hence, $e^{t\Delta}\phi$ is convex in ${\mathbb R}^n$ for all $t\in(0,T(\phi))$. 
  \item[(A2)] 
  Assume that $\phi$ is log-convex in ${\mathbb R}^n$, that is, 
  $$
  \phi((1-\lambda)x_0+\lambda x_1)\le \phi(x_0)^{1-\lambda}\phi(x_1)^\lambda,\quad x_0,x_1\in{\mathbb R}^n,\,\,\,\lambda\in(0,1).
  $$
  Then  {H\"older's inequality yields}
  \begin{equation*}
  \begin{split}
   & \left[e^{t\Delta}\phi\right]((1-\lambda)x_0+\lambda x_1)\\
   & \le \int_{{\mathbb R}^n} {\Gamma_n}(y,t)\phi(x_0-y)^{1-\lambda}\phi(x_1-y)^\lambda\,\dee y\\
   & \le \left(\int_{{\mathbb R}^n} {\Gamma_n}(y,t)\phi(x_0-y)\,\dee y\right)^{1-\lambda}\left(\int_{{\mathbb R}^n} {\Gamma_n}(y,t)\phi(x_1-y)\,\dee y\right)^\lambda\\
   & =\left[e^{t\Delta}\phi\right](x_0)^{1-\lambda}\left[e^{t\Delta}\phi\right](x_1)^\lambda,
   \quad x_0,x_1\in{\mathbb R}^n,\,\,\,\lambda\in(0,1),\,\,\,t\in(0,T(\phi)).
  \end{split}
  \end{equation*}
 Hence, $e^{t\Delta}\phi$ is log-convex in ${\mathbb R}^n$ for all $t\in(0,T(\phi))$. 
\end{itemize}
The preservation of convexity for solutions of linear parabolic equations on ${\mathbb R}^n$ with at most polynomial growth
was studied in detail in~\cite{JT}.

On the other hand, in \cite{IST05},
the first and second authors of the present paper, together with Takatsu,
introduced the notion of $F$-concavity and characterized $F$-concavities  {preserved under} the Dirichlet heat flow.
In particular, they proved that, for nonnegative functions, among $F$-concavity properties, only log-concavity is preserved 
 {under} the heat flow when $n\ge 2$ (see \cite{IST05}*{Theorem~1.5}).
In the one-dimensional case $n=1$, only log-concavity and quasi-concavity are preserved  {under} the heat flow (see \cite{IST05}*{Theorem~1.6}). 
These results, together with (A1) and (A2), suggest that the variety of convexity properties preserved  {under} the heat flow, 
when considering nonnegative solutions, is richer than that of concavity properties.

In this paper, we refine and adapt the arguments in~\cite{IST05} to introduce the notion of $F$-convexity,
which provides a general extension of the classical concepts of convexity and log-convexity.
We then characterize $F$-convexities  {preserved under} the heat flow.
Furthermore, we  {introduce} an order structure that allows us to compare the strength of different
$F$-convexity properties, and, under a suitable restriction on $F$, we identify both the strongest 
and the weakest $F$-convexities  {preserved under} the heat flow.
\medskip

We begin by introducing the notion of $F$-convexity for  {nonnegative} functions.
Throughout this paper, we adhere to the convention that
\begin{align*}
 & -\infty\leq -\infty,\qquad
 \infty\le\infty,\qquad
 -\infty+b
=b-\infty=-\infty, \qquad
\kappa \cdot(\pm\infty)=\pm\infty,\\
 & \,e^{-\infty}=0,\qquad \log 0=-\infty,\qquad -\log 0=\infty,\qquad \log\infty=\infty,
\end{align*}
where $b\in [-\infty,\infty)$ and $\kappa\in(0,\infty)$. 
\begin{definition}
\label{Definition:1.1}
A function $F:[0,\infty)\rightarrow{\mathbb R}\cup\{-\infty\}$ is {\em admissible} on $[0,\infty)$ if 
$F$ is strictly increasing on $[0,\infty)$, continuous on $(0,\infty)$, and $F(0)=\lim_{r\to 0^+}F(r)$.
Throughout the paper, we denote by $f_F$ the inverse function of $F$ on $J_F:=F((0,\infty))$. 
\end{definition}
\begin{definition}
\label{Definition:1.2}
Let $F$ be admissible on $[0,\infty)$. 
\begin{itemize}
  \item[{\rm (1)}] 
  A nonnegative function $f$ on ${\mathbb R}^n$ is said to be {\em $F$-convex} if $F(f)$ is convex in ${\mathbb R}^n$; that is, 
  $F(f)$ satisfies 
  $$
  F(f((1-\lambda)x+\lambda y))\le(1-\lambda)F(f(x))+\lambda F(f(y))
  $$
  for all $x$, $y\in{\mathbb R}^n$ and $\lambda\in(0,1)$. 
  We denote by $\mathcal{C}[F]$ the set of $F$-convex {\rm({\it nonnegative})} functions on ${\mathbb R}^n$.
  For any admissible functions $F_1$ and $F_2$, if ${\mathcal C}[F_1]\subset{\mathcal C}[F_2]$, 
  we say that $F_1$-convexity is stronger than $F_2$-convexity 
  or equivalently, that $F_2$-convexity is weaker than $F_1$-convexity. 
  \item[{\rm (2)}] 
  $F$-convexity is said to be {\em trivial} if ${\mathcal C}[F]$ consists only of nonnegative  {constant} functions on~${\mathbb R}^n$. 
  \item[{\rm (3)}] 
  We say that $F$-convexity is {\em preserved}  {under} the heat flow if, for any $F$-convex {\rm ({\it nonnegative})} function $\phi$ with 
  $T(\phi)\in(0,\infty]$, the function $e^{t\Delta}\phi$ remains $F$-convex for all $t\in(0,T(\phi))$.
  Furthermore, we say that $F$-convexity is {\em only trivially preserved}  {under} the heat flow if the following holds:
  whenever $\phi$ is $F$-convex on ${\mathbb R}^n$, satisfies $T(\phi)\in(0,\infty]$, and $e^{t\Delta}\phi$ is 
  $F$-convex for all $t\in(0,T(\phi))$, then $\phi$ must be constant on ${\mathbb R}^n$.
 \end{itemize}
\end{definition}
\begin{remark}
\label{Remark:1.1}
Let $F$ be admissible on $[0,\infty)$ and assume that
\[
\lim_{r\to\infty} F(r) < \infty .
\]
Then $F$-convexity is trivial.
Indeed, since any convex function  {on ${\mathbb R}^n$ that is bounded above}  is constant,
if $F(\phi)$ is convex on ${\mathbb R}^n$ for some nonnegative function $\phi$ on ${\mathbb R}^n$,
then $\phi$ must be constant on ${\mathbb R}^n$.
\end{remark}

We  {next} define $\alpha$-convexity. 
\begin{definition}
\label{Definition:1.3}
Let $\alpha\in[-\infty,\infty]$. 
\begin{itemize}
  \item[{\rm (1)}]
  For any $a,b>0$ and $\lambda\in(0,1)$, define 
  \[
  {\mathcal M}_{\alpha}(a,b;\lambda)
  :=\left\{
  \begin{array}{ll}
  \max\{a,b\} & \text{if } \alpha=\infty,\\[4pt]
  \left((1-\lambda)a^{\alpha}+\lambda b^{\alpha}\right)^{\frac{1}{\alpha}}
  & \text{if } \alpha\in\mathbb{R}\setminus\{0\},\\[4pt]
  a^{1-\lambda}b^\lambda & \text{if } \alpha=0,\\[4pt]
  \min\{a,b\} & \text{if } \alpha=-\infty.
  \end{array}
  \right.
  \]
  Furthermore, for $a,b\ge 0$, we define ${\mathcal M}_{\alpha}(a,b;\lambda)$ as above if $\alpha\ge 0$, and
  \[
  {\mathcal M}_{\alpha}(a,b;\lambda)=0
  \quad \text{if } \alpha<0 \text{ and } a\,b=0.
  \]

  \item[{\rm (2)}]
  A nonnegative function $f$ on ${\mathbb R}^n$ is said to be $\alpha$-convex if
  \[
  f((1-\lambda)x+\lambda y)\le {\mathcal M}_{\alpha}(f(x),f(y);\lambda)
  \]
  for all $x,y\in{\mathbb R}^n$ and $\lambda\in(0,1)$.
  We denote by ${\mathcal C}[\alpha]$ the set of $\alpha$-convex functions on ${\mathbb R}^n$.
  Note that the case $\alpha=1$ corresponds to the usual convexity, while $\alpha=0$ corresponds to log-convexity.
\end{itemize}  
\end{definition}
We refer to $\infty$-convexity as \emph{quasi-convexity}: it is easily seen that a function is {\em quasi-convex} if and only if all its sublevel sets are convex.
The term \emph{power convexity} is used as a general term for $\alpha$-convexity. 
 {We next present several remarks on power convexity.}
\begin{itemize}
  \item 
  For each $\alpha\in{\mathbb R}$, 
  define a function $\Phi_\alpha$ on $[0,\infty)$ by
  \[
  \Phi_\alpha(r):=\int_1^r s^{\alpha-1}\,\mathrm{d}s
  =\left\{
  \begin{array}{ll}
  \dfrac{r^\alpha-1}{\alpha} & \text{if } \alpha\not=0,\\[6pt]
  \log r  & \text{if } \alpha=0,
  \end{array}
  \right.
  \]
  for $r\in(0,\infty)$, while 
  \[
  \Phi_\alpha(0):=
  \left\{
  \begin{array}{ll}
  \dfrac{-1}{\alpha} & \text{if } \alpha>0,\\[6pt]
-\infty & \text{if } \alpha\leq0.
  \end{array}
  \right.
  \]
  Then $\Phi_\alpha$ is admissible on $[0,\infty)$ and ${\mathcal C}[\Phi_\alpha]={\mathcal C}[\alpha]$. 
  Hence, $\alpha$-convexity is a particular case of $F$-convexity. 
  When $\alpha\in(-\infty,0)$, 
  it follows from Remark~\ref{Remark:1.1} that $\alpha$-convexity is trivial, 
  that is, ${\mathcal C}[\alpha]$ consists only of nonnegative constant functions.
  \item
  Any $(-\infty)$-convex function on ${\mathbb R}^n$ must be constant on ${\mathbb R}^n$ {; hence,} $(-\infty)$-convexity is trivial.
  \item 
  Although quasi-convexity does not admit any corresponding admissible function on $[0,\infty)$
  (see, e.g., \cite{IST03}*{Remark~2.2} and \cite{CR}), 
   {the monotonicity of admissible functions implies that} 
  any $F$-convex function is necessarily quasi-convex. 
  \item 
  If $-\infty\le \alpha\le\beta\le \infty$, then $\alpha$-convexity is stronger than $\beta$-convexity; 
  equivalently, $\beta$-convexity is weaker than $\alpha$-convexity. 
   {This leads to the following hierarchy:}
\end{itemize}
$$
\underset{(\text{trivial})}{\text{$(-\infty)$-convexity}}=\underset{(\text{trivial for $\alpha<0$})}{\text{negative power convexity}}\subsetneq
\underset{(\alpha=0)}{\text{log-convexity}}
\subsetneq \underset{(\alpha=1)}{\text{convexity}}\subsetneq
\underset{(\alpha=\infty)}{\text{quasi-convexity}}
$$

We are now ready to state the main results of this paper. 
In the first theorem, we provide a necessary and sufficient condition 
for $F$-convexity to be preserved  {under} the heat flow. 
\begin{theorem}
\label{Theorem:1.1}
Let $F$ be admissible on $[0,\infty)$.
\begin{itemize}
  \item[{\rm (1)}] 
  If $\lim_{r\to\infty}F(r)<\infty$, then $F$-convexity is only trivially preserved  {under} the heat flow.
  \item[{\rm (2)}] 
  Let $\lim_{r\to\infty}F(r)=\infty$. 
  Assume that 
  \begin{equation}
  \label{eq:1.3}
  \int_{F(1)}^\infty e^{-A z^2}f_F(z)\,\mathrm{d}z=\infty
  \quad \text{for any } A>0.
  \end{equation}
  Then $F$-convexity is only trivially preserved  {under} the heat flow. 
  \item[{\rm (3)}] 
  Let $F\in C^2((0,\infty))$ and $\lim_{r\to\infty}F(r)=\infty$. 
  Assume that 
  \begin{equation}
  \tag{F}
  \label{eq:F}
  \int_{F(1)}^\infty e^{-A z^2}f_F(z)\,\mathrm{d}z<\infty
  \quad \text{for some } A>0.
  \end{equation}
  Then $F$-convexity is preserved  {under} the heat flow if and only if 
  \begin{equation}
  \label{eq:1.4}
  F'>0 \text{ in } (0,\infty),
  \qquad 
  g_F:=(\log f_F')' \text{ is convex on } J_F. 
  \end{equation}
\end{itemize}
\end{theorem}
By Theorem~\ref{Theorem:1.1}, 
$\alpha$-convexity with $\alpha\in{\mathbb R}$ is preserved  {under} the heat flow if and only if $\alpha\le 1$, 
while it is only trivially preserved when $-\infty<\alpha<0$.
Note that $\infty$-convexity is preserved  {under} the heat flow if and only if $n=1$ (see \cite{IS01}).

Next, under a suitable restriction on $F$, 
we apply Theorem~\ref{Theorem:1.1}~(3) 
to identify the weakest and the strongest $F$-convexities 
preserved  {under} the heat flow. 
\begin{theorem}
\label{Theorem:1.2}
Let $F\in C^2((0,\infty))$ be admissible  {on} $[0,\infty)$ and $\lim_{r\to\infty}F(r)=\infty$. 
Assume that 
\begin{equation}
  \tag{F'}
  \label{eq:F'}
  \int_{F(1)}^\infty e^{-A z^2}f_F(z)\,\mathrm{d}z<\infty
  \quad \text{for any } A>0.
\end{equation}
If $F$-convexity is preserved  {under} the heat flow, then 
$F$-convexity is stronger than $1$-convexity and weaker than log-convexity. 
\end{theorem}
As a direct consequence of Theorem~\ref{Theorem:1.2}, we have:
\begin{corollary}
\label{Corollary:1.1}
Let $F\in C^2((0,\infty))$ be admissible  {on $[0,\infty)$} and  satisfy $\lim_{r\to\infty}F(r)=\infty$ and condition~\eqref{eq:F'}. 
Then $1$-convexity {\rm({\it resp.\ log-convexity})} 
is the weakest {\rm ({\it resp.\ strongest})} nontrivial $F$-convexity preserved  {under} the heat flow.
\end{corollary}
\begin{remark}
\label{Remark:1.2}
{\rm (1)} 
Condition~\eqref{eq:F'} is equivalent to the existence of nontrivial $F$-convex initial data 
for which the heat flow exists globally in time. 
\vspace{3pt}
\newline
{\rm (2)}
Let $g$ be the following convex function on $[0,\infty)$:
\[
g(z):=|z-1|-1,\quad z\in[0,\infty). 
\]
Define
\[
f(z):=\int_0^z\exp\!\left(\int_0^s g(\tau)\,\mathrm{d}\tau\right)\mathrm{d} {s},
\quad 
F(r):=f^{-1}(r),
\]
for $z$, $r\in[0,\infty)$. 
Then $F$ is admissible on $[0,\infty)$ and satisfies conditions~\eqref{eq:F} and \eqref{eq:1.4}.
Hence, by Theorem~{\rm\ref{Theorem:1.1}}, $F$-convexity is preserved  {under} the heat flow. 

On the other hand, $f''$ is negative in a neighborhood of $z=1$. 
This implies that $F$-convexity is not stronger than usual convexity. 
Consequently, Theorem~{\rm\ref{Theorem:1.2}} does not hold without any restriction such as condition~\eqref{eq:F'}.
\end{remark}
\begin{remark}
\label{Remark:1.3} 
{\rm (1)}  
Let $F$ be  {an} admissible function on $[0,\infty)$ and $u$ be a nonnegative function on~${\mathbb R}^n$. 
Set $v=e^{-u}$. Then the $F$-convexity of $u$ is equivalent to the $G$-concavity of $v$, where 
$$G(s)=-F(-\log(s))\,.$$
In particular, the $\alpha$-convexity of $u$ corresponds to the $\alpha$-log-concavity of $v$.

Now, let $\phi$ be a continuous positive function  {on} ${\mathbb R}^n$ such that $T(\phi)\in(0,\infty]$. 
Set 
\[
u(x,t)=[e^{t\Delta}\phi](x),\quad
v(x,t)=\exp(-u(x,t)), 
\]
for $(x,t)\in{\mathbb R}^n\times[0,T(\phi))$. 
Then, the preservation of $F$-convexity  {under} the heat flow with positive initial data is reduced to the preservation of $G$-concavity
 {for the parabolic equation} 
\begin{equation}
\label{eq:1.5}
\partial_t v=\Delta v-\frac{|\nabla v|^2}{v}\quad\text{in}\quad{\mathbb R}^n\times(0,T(\phi)).
\end{equation}
On the other hand, $\alpha$-log-concavity is preserved  {under} the heat flow if and only if $0 \le \alpha \le  {1/2}$ {\rm ({\it see} \cites{IST01, IST03, IST05})},
whereas $\alpha$-convexity is preserved  {under} the heat flow if and only if $ {0} \le \alpha \le 1$.
This difference originates from the nonlinear term $-|\nabla v|^2/v$ in equation~\eqref{eq:1.5}.
\vspace{3pt}
\newline
{\rm (2)} 
The results in \cite{IST05}*{Theorem~1.5~(A4), Proposition~3.1} suggest that, 
if $n\ge 2$ and $F$-convexity is strictly weaker than the usual convexity, then there exists an $F$-convex function $\phi$ on ${\mathbb R}^n$ 
for which $T(\phi)\in(0,\infty]$ and $e^{t\Delta}\phi$ is not $F$-convex on ${\mathbb R}^n$ for some $t\in(0,T(\phi))$. 
However, this problem remains open. 
\end{remark}

We prove Theorems~\ref{Theorem:1.1} and \ref{Theorem:1.2} in Section~3. 
In Section~4, we adapt the arguments in the proofs of Theorems~\ref{Theorem:1.1} and~\ref{Theorem:1.2} 
to study the preservation of $F$-convexity under the heat flow without any sign restriction on the solutions.
In particular, we show that an analogue of Theorem~\ref{Theorem:1.1} holds for functions bounded from below.
In this setting, under a suitable growth condition on $f_F$ analogous to~(F'), we prove that the only $F$-convexity preserved under the heat flow is $1$-convexity.
Finally, in Section~5, we apply the arguments in~\cite{IST05} to study the preservation of $F$-convexity under the Dirichlet heat flow in convex domains.
\medskip

We outline the main ideas of the proofs of Theorems~\ref{Theorem:1.1} and \ref{Theorem:1.2}.
In~\cite{IST05}, the preservation of $F$-concavity  {under} the Dirichlet heat flow with bounded initial data was studied. 
The boundedness of the initial data is not restrictive in the study of the preservation of $F$-concavity  {under} the Dirichlet heat flow. 
Indeed, in the case of the heat flow, it follows that
\[
\left[e^{t\Delta}\phi\right](x)
=\lim_{k\to\infty}\int_{{\mathbb R}^n}\Gamma_n(x-y,t)\min\{\phi(y),k\}\,\mathrm{d}y
\]
for any nonnegative function $\phi$ on ${\mathbb R}^n$. 
Furthermore, if $\phi$ is $F$-concave on ${\mathbb R}^n$, then $\min\{\phi,k\}$ is also $F$-concave on ${\mathbb R}^n$ for any $k>0$.
Letting $k\to\infty$, we can study the preservation of $F$-concavity under the heat flow without assuming the boundedness of the initial data.
In contrast, if $\phi$ is $F$-convex on ${\mathbb R}^n$, $\min\{\phi,k\}$ is not necessarily $F$-convex on ${\mathbb R}^n$ for $k>0$.
This means that the arguments in \cite{IST05} are not directly applicable to the study of the preservation of convexity properties  {under} the heat flow. 

In this paper, we modify the arguments in \cite{INS}. 
More precisely, for any $\lambda\in(0,1)$ and any nonnegative continuous function $\phi$ on ${\mathbb R}^n$ satisfying
\[
T(\phi)\in(0,\infty],\qquad
\lim_{|x|\to\infty}\left[e^{t\Delta}\phi\right](x)=\infty
\quad \text{for } t\in(0,T(\phi)),
\]
 {we show that} 
the function $U_\lambda$ defined by 
\begin{align*}
U_\lambda(x,t):=\inf\Bigl\{ F^{-1}\bigl((1-\lambda)F([e^{t\Delta}\phi](x_0))
+\lambda F([e^{t\Delta}\phi](x_1))\bigr) & \\
\, : \,
x=(1-\lambda)x_0+\lambda x_1,\, x_0,x_1\in{\mathbb R}^n \Bigr\} & 
\end{align*}
is a viscosity supersolution to the heat equation in ${\mathbb R}^n\times(0,T(\phi))$ under condition~\eqref{eq:1.4}. 
Then, by the comparison principle for viscosity solutions of the heat equation and a suitable approximation of the initial data, 
we prove the preservation of $F$-convexity under the heat flow assuming conditions~\eqref{eq:F} and \eqref{eq:1.4}. 
 {To this end}, it is necessary to obtain growth estimates of solutions at spatial infinity. 
On the other hand, we apply \cite{IST05}*{Proposition~4.8} (see also \cite{Kol}) to prove that 
if $F$-convexity is preserved  {under} the heat flow, then $F$ satisfies condition~\eqref{eq:1.4}. 
 {This argument yields} the proof of Theorem~\ref{Theorem:1.1}. 
Theorem~\ref{Theorem:1.2} follows from Theorem~\ref{Theorem:1.1}. 
\medskip

The rest of this paper is organized as follows.
In Section~2, we recall several lemmas on $F$-convexity and on viscosity solutions, 
and prove a key proposition  {used in} the proof of Theorem~\ref{Theorem:1.1}.
In Section~3, we prove Theorems~\ref{Theorem:1.1} and \ref{Theorem:1.2}.  
In Section~4, we study the preservation of $F$-convexity under the heat flow without any sign restriction on the solutions.
In Section~5, we apply the arguments in \cite{IST05} to study the $F$-convexities preserved under the Dirichlet heat flow.
%%%%%%%%%%%%%%%%%%%%%%
%%%%%%%%%%%%%%%%%%%%%%
\section{Preliminaries}
%%%%%%%%%%%%%%%%%%%%%%
%%%%%%%%%%%%%%%%%%%%%%
We first recall two lemmas on the order structure of $F$-convexities. 
\begin{lemma}
\label{Lemma:2.1}
Let $F_1$ and $F_2$ be admissible on $[0,\infty)$. 
Then $F_1$-convexity is weaker than $F_2$-convexity 
if and only if $F_1\circ f_{F_2}$ is convex on $J_{F_2}$ 
{\rm (or, equivalently, $F_2\circ f_{F_1}$ is concave on $J_{F_1}$)}.
\end{lemma}
\begin{lemma}
\label{Lemma:2.2}
Let $F_1$ and $F_2$ be admissible on $[0,\infty)$. 
Then the relation ${\mathcal C}[F_1]={\mathcal C}[F_2]$ holds if and only if 
there exists a pair $(A,B)\in (0,\infty)\times{\mathbb R}$ such that 
\[
F_1(r)=AF_2(r)+B\quad\text{for}\quad r\in (0,\infty).
\]
\end{lemma}
The proofs of Lemmas~\ref{Lemma:2.1} and \ref{Lemma:2.2} are the same as in \cite{IST05}*{Lemmas 2.4 and 2.5}, respectively. 
We omit the proofs.

Next, we consider a nonlinear parabolic equation
\begin{equation}
\tag{E}
\label{eq:E} 
\partial_t u=H(x,t,u,\nabla u,\nabla^2 u)
\quad\text{in}\quad{\mathbb R}^n\times(0,T), 
\end{equation}
where $T\in(0,\infty]$.
We impose the following assumptions on $H$:
\begin{itemize}
  \item[(H1)] 
  $H\in C({\mathbb R}^n\times(0,T)\times{\mathbb R}\times{\mathbb R}^n\times{\mathbb S}_n)$;

  \item[(H2)] 
  $H$ is a degenerate elliptic operator, that is, the map
  \[
   {{\mathbb S}_n}\ni A\mapsto H(x,t,v,\theta,A)
  \]
  is non-decreasing for any fixed $(x,t,v,\theta)\in {\mathbb R}^n\times(0,T)\times{\mathbb R}\times{\mathbb R}^n$.
\end{itemize}
Here ${\mathbb S}_n$ denotes the set of all real $n\times n$ symmetric matrices.

We recall the notions of viscosity subsolution, supersolution, and solution to equation~\eqref{eq:E}.
An upper semicontinuous function $v$ on ${\mathbb R}^n\times(0,T)$ 
is called a \emph{viscosity subsolution} of equation~\eqref{eq:E}
if, for any $(\xi,\tau)\in{\mathbb R}^n\times(0,T)$, 
one has
$$
\partial_t\psi(\xi,\tau)-H(\xi,\tau,\psi(\xi,\tau),\nabla\psi(\xi,\tau),\nabla^2\psi(\xi,\tau))\le 0
$$
for every test function $\psi\in C^{2;1}({\mathbb R}^n\times(0,T))$ that touches $v$ from above at $(\xi,\tau)$, that is,
$$
\psi(\xi,\tau)=v(\xi,\tau),\qquad  \text{$\psi\ge v$ in a neighborhood of $(\xi,\tau)$}.
$$
Analogously, a lower semicontinuous function $v$ on ${\mathbb R}^n\times(0,T)$ 
is called a \emph{viscosity supersolution} of equation~\eqref{eq:E} 
if, for any $(\xi,\tau)\in{\mathbb R}^n\times(0,T)$, 
one has 
$$
\partial_t\psi(\xi,\tau)-H(\xi,\tau,\psi(\xi,\tau),\nabla\psi(\xi,\tau),\nabla^2\psi(\xi,\tau))\ge 0
$$
for every test function $\psi\in C^{2;1}({\mathbb R}^n\times(0,T))$ that touches $v$ from below at $(\xi,\tau)$, that is, 
$$
\psi(\xi,\tau)=v(\xi,\tau),\qquad \text{$\psi\le v$ in a neighborhood of $(\xi,\tau)$.}
$$
A continuous function $v$ on ${\mathbb R}^n\times(0,T)$ is called a \emph{viscosity solution} of equation~\eqref{eq:E}
if it is both a~viscosity subsolution and supersolution.

Next, we prove a comparison principle for viscosity solutions of the heat  {equation}, which will be used in the next sections.
\begin{lemma}
\label{Lemma:2.3}
Let $T>0$. 
Let $u\in C({\mathbb R}^n\times[0,T))$ be a viscosity subsolution of the heat equation in ${\mathbb R}^n\times(0,T)$ 
such that $u\le 0$ on ${\mathbb R}^n\times\{0\}$. 
Assume that 
\begin{equation}
\label{eq:2.1}
u(x,t)\le a e^{A|x|^2}\quad\text{in}\quad{\mathbb R}^n\times[0,T)
\end{equation}
for some $A>0$ and $a>0$.
Then $u\le 0$ in ${\mathbb R}^n\times[0,T)$. 
\end{lemma}
{\bf Proof.}
The proof is similar to that of \cite{E}*{Section~2.3, Theorem~6}. 
However, for the sake of completeness, we include the argument here.

Assume that $4AT<1$. 
Let $\epsilon>0$ be such that $4A(T+\epsilon)<1$. 
For any fixed $y\in{\mathbb R}^n$ and $\mu>0$, define
\[
v(x,t):=\mu(T+\epsilon-t)^{-\frac{n}{2}}
\exp\!\left(\frac{|x-y|^2}{4(T+\epsilon-t)}\right)
\]
for $(x,t)\in{\mathbb R}^n\times[0,T)$, which is a solution of the heat equation in ${\mathbb R}^n\times(0,T)$. 
Then the function~$w$ defined by 
\[
w(x,t):=u(x,t)-v(x,t),
\quad (x,t)\in{\mathbb R}^n\times[0,T)
\]
is a viscosity subsolution of the heat equation in ${\mathbb R}^n\times(0,T)$. 
Let $r>0$. 
By \eqref{eq:2.1}, 
for $(x,t)\in \partial B(y,r)\times(0,T)$, we have
\begin{align*}
w(x,t)
&\le a e^{A|x|^2}
-\mu(T+\epsilon-t)^{-\frac{n}{2}}
\exp\!\left(\frac{r^2}{4(T+\epsilon-t)}\right)\\
&\le a e^{A(|y|+r)^2}
-\mu(T+\epsilon)^{-\frac{n}{2}}
\exp\!\left(\frac{r^2}{4(T+\epsilon)}\right).
\end{align*}
Set
\[
\gamma:=\frac{1}{4(T+\epsilon)}-A>0.
\]
Then, taking $r$ sufficiently large, we obtain
$$ 
w(x,t)\le ae^{A(|y|+r)^2}-\mu(4(A+\gamma))^{\frac{n}{2}}e^{(A+\gamma)r^2}\le 0
$$
for $(x,t)\in\partial B(y,r)\times(0,T)$. 
We now apply the comparison principle (see, e.g., \cite{CIL}*{Theorem~8.2}) for the heat equation in $B(y,r)\times[0,T)$ 
to conclude that $w\le 0$ in $B(y,r)\times[0,T)$, that is,
\[
u(x,t)\le \mu(T+\epsilon-t)^{-\frac{n}{2}}
\exp\!\left(\frac{|x-y|^2}{4(T+\epsilon-t)}\right),
\quad (x,t)\in B(y,r)\times[0,T).
\]
Setting $x=y$, we obtain
\[
u(y,t)\le \mu(T+\epsilon-t)^{-\frac{n}{2}},
\quad t\in[0,T).
\]
Letting $\mu\to 0^+$, we conclude that 
$u(y,t)\le 0$ for all $t\in[0,T)$. 
Since $y\in{\mathbb R}^n$ is arbitrary, we obtain $u\le 0$ in ${\mathbb R}^n\times[0,T)$. 
Thus Lemma~\ref{Lemma:2.3} follows in the case where $4AT<1$.

In the case where $4AT\ge 1$, 
we repeatedly apply the argument above on the time intervals $[0,T_1)$, $[T_1,2T_1)$, $\dots$, 
where $4AT_1<1$, and obtain $u\le 0$ in ${\mathbb R}^n\times[0,T)$. 
This completes the proof of Lemma~\ref{Lemma:2.3}.
$\Box$
%
%%%%%%%%%%%%%%%%%%%%%%
%%%%%%%%%%%%%%%%%%%%%%
\section{Proof of Theorems~\ref{Theorem:1.1} and \ref{Theorem:1.2}}
%%%%%%%%%%%%%%%%%%%%%%
%%%%%%%%%%%%%%%%%%%%%%
 {Throughout the rest of the paper, let $F$ be admissible on $[0,\infty)$, unless otherwise stated.}

 {We first prove a proposition concerning} solutions to equation~\eqref{eq:E}, which is crucial for the proof of Theorem~\ref{Theorem:1.1}.
\begin{proposition}
\label{Proposition:3.1}
Assume conditions~{\rm (H1)} and {\rm (H2)}. 
Let $u$ be a classical solution of equation~\eqref{eq:E}, that is, 
$u\in C^{2,1}({\mathbb R}^n\times(0,T))\cap C({\mathbb R}^n\times[0,T))$ and $u$ satisfies equation~\eqref{eq:E} pointwise, 
such that 
\begin{equation}
\label{eq:3.1}
\lim_{|x|\to\infty}u(x,t)=\infty,\quad t\in(0,T). 
\end{equation}
For any $\lambda\in(0,1)$, define 
\begin{equation}
\label{eq:3.2}
U_\lambda(x,t):=\inf\Bigl\{(1-\lambda)u(x_0,t)+\lambda u(x_1,t)
:\,x=(1-\lambda)x_0+\lambda x_1,\,x_0,x_1\in{\mathbb R}^n\Bigr\}
\end{equation}
for $(x,t)\in{\mathbb R}^n\times[0,T)$. 
Then $U_\lambda$ is a viscosity supersolution of equation~\eqref{eq:E} if the following property holds.
\begin{itemize}
  \item[{\rm (H3)}] 
  For any fixed $(t,\theta)\in(0,T)\times{\mathbb R}^n$, 
  the function 
  \[
  {\mathbb R}^n\times{\mathbb R}\times{\mathbb S}_n\ni (x,v,A)\mapsto H(x,t,v,\theta,A)
  \]
  is convex. 
\end{itemize}
\end{proposition}
{\bf Proof.}
Let $\lambda\in(0,1)$. 
It follows from \eqref{eq:3.1} that $U_\lambda\in C({\mathbb R}^n\times[0,T))$. 
Furthermore, for any $(x^*,t^*)\in{\mathbb R}^n\times(0,T)$, 
there exist $x_0^*,x_1^*\in{\mathbb R}^n$ such that 
\begin{equation}
\label{eq:3.3}
x^*=(1-\lambda)x_0^*+\lambda x_1^*,\quad
U_*:= {U_\lambda}(x^*,t^*)=(1-\lambda)u_0^*+\lambda u_1^*,
\end{equation}
where $u_0^*:=u(x_0^*,t^*)$ and $u_1^*:=u(x_1^*,t^*)$. 
Then, by the method of Lagrange multipliers, 
there exists $\theta\in{\mathbb R}^n$ such that 
\begin{equation}
\label{eq:3.4}
(\nabla u)(x_0^*,t^*)=(\nabla u)(x_1^*,t^*)=\theta. 
\end{equation}
Define a function $\varphi$ on ${\mathbb R}^n\times(0,T)$ by
\[
\varphi(x,t):=(1-\lambda)u(x_0^*+x-x^*,t)
+\lambda u(x_1^*+x-x^*,t),
\quad (x,t)\in{\mathbb R}^n\times(0,T).
\]
It follows from \eqref{eq:3.3} that
\[
(1-\lambda)(x_0^*+x-x^*)+\lambda(x_1^*+x-x^*)
=(1-\lambda)x_0^*+\lambda x_1^*+x-x^*=x,
\quad x\in{\mathbb R}^n,
\]
which, together with \eqref{eq:3.2} and \eqref{eq:3.4}, implies that 
\begin{equation}
\label{eq:3.5}
\left\{
\begin{aligned}
 & U_\lambda(x,t)\le\varphi(x,t)\quad \text{for } (x,t)\in{\mathbb R}^n\times(0,T),\\
 & \varphi(x^*,t^*)=(1-\lambda)u(x_0^*,t^*)+\lambda u(x_1^*,t^*)=U_*,\\
 & (\nabla\varphi)(x^*,t^*)
 =(1-\lambda)(\nabla u)(x_0^*,t^*)
 +\lambda(\nabla u)(x_1^*,t^*)
 =\theta.
\end{aligned}
\right.
\end{equation}
Then, by condition~(H3), \eqref{eq:3.3}, and \eqref{eq:3.5}, 
we obtain 
\begin{align*}
 & (\partial_t\varphi)(x^*,t^*)
 -H\bigl(x^*,t^*,\varphi(x^*,t^*),
 (\nabla\varphi)(x^*,t^*),
 (\nabla^2\varphi)(x^*,t^*)\bigr)\\
 & =(1-\lambda)(\partial_t u)(x_0^*,t^*)
 +\lambda(\partial_t u)(x_1^*,t^*)\\
 & \qquad
 -H\bigl(x^*,t^*,(1-\lambda)u_0^*+\lambda u_1^*,
 (1-\lambda)\theta+\lambda\theta,
 (1-\lambda)A_0+\lambda A_1\bigr)\\
 & =(1-\lambda)H(x_0^*,t^*,u_0^*,\theta,A_0)
 +\lambda H(x_1^*,t^*,u_1^*,\theta,A_1)\\
 & \qquad
-H\bigl(x^*,t^*,(1-\lambda)u_0^*+\lambda u_1^*,
 \theta,(1-\lambda)A_0+\lambda A_1\bigr)\ge 0,
\end{align*}
where $A_0:=(\nabla^2 u)(x_0^*,t^*)$ and 
$A_1:=(\nabla^2 u)(x_1^*,t^*)$. 
Consequently, $U_\lambda$ is a viscosity supersolution of equation~\eqref{eq:E}, 
and the proof is complete.
$\Box$
\medskip

Now we are ready to prove Theorem~\ref{Theorem:1.1}.
\medskip

\noindent
{\bf Proof of Theorem~\ref{Theorem:1.1}.}
Assertion~(1) follows immediately from Remark~\ref{Remark:1.1}. 

We prove assertion~(2). 
Let $F$ be admissible on $[0,\infty)$ such that $\lim_{r\to\infty}F(r)=\infty$. 
Then $J_F=F((0,\infty))=(F(0^+),\infty)$, where $F(0^+):=\lim_{r\to 0^+}F(r)$.
Assume \eqref{eq:1.3}. 
Let $\phi$ be a  {nonconstant $F$-convex function} on ${\mathbb R}^n$. 
Since $F(\phi)$ is not constant and is continuous on ${\mathbb R}^n$, 
there exist $x_0\in{\mathbb R}^n$, $\theta_0\in{\mathbb S}^{n-1}$, $r_0>0$, $\delta_0>0$, and $\kappa>0$ such that 
\[
F(\phi(x_0+\theta r_0))\ge F(\phi(x_0))+\kappa
\]
for all $\theta\in{\mathbb S}^{n-1}$ with $|\theta-\theta_0|<\delta_0$. 
Since $F(\phi)$ is convex on ${\mathbb R}^n$, we obtain
\[
F(\phi(x_0+\theta r))
\ge F(\phi(x_0))+\kappa\frac{r}{r_0}
\]
for all $\theta\in{\mathbb S}^{n-1}$ with $|\theta-\theta_0|<\delta_0$ and all $r\ge r_0$. 
This implies
\[
\phi(x_0+\theta r)\ge 
f_F\!\left(F(\phi(x_0))+\kappa\frac{r}{r_0}\right)
\]
for all such $\theta$ and $r$. 
Then it follows from \eqref{eq:1.3} that
\[
\int_{{\mathbb R}^n} e^{-A|x|^2}\phi(x)\,\mathrm{d}x=\infty
\quad \text{for any } A>0,
\]
which, together with \eqref{eq:1.1}, implies that $T(\phi)=0$. 
Hence, $F$-convexity is only trivially preserved  {under} the heat flow, and assertion~(2) holds.

We prove assertion~(3).
Let $F\in C^2((0,\infty))$ be admissible on $[0,\infty)$ such that $\lim_{r\to\infty}F(r)=\infty$. 
Assume condition~\eqref{eq:F}.
Let $\phi$ be $F$-convex on ${\mathbb R}^n$, and assume that $T(\phi)\in(0,\infty]$. 
Set 
\begin{equation}
\label{eq:3.6}
u(x,t):=\left[e^{t\Delta}\phi\right](x),\quad v(x,t):=F(u(x,t)),
\end{equation}
for $x\in{\mathbb R}^n$ and $t\in[0,T(\phi))$.
For any $\epsilon>0$, set 
\begin{equation}
\label{eq:3.7}
\begin{split}
 & \phi_\epsilon(x):=\phi(x)+\epsilon |x|^2,\\
 & u_\epsilon(x,t):=\left[e^{t\Delta}\phi_\epsilon\right](x)=u(x,t)+\epsilon(|x|^2+2nt),\\
 & v_\epsilon(x,t):=F(u_\epsilon(x,t)),
\end{split}
\end{equation}
for $x\in{\mathbb R}^n$ and $t\in[0,T(\phi))$. 
Since $\lim_{r\to\infty}F(r)=\infty$, we have
\begin{equation}
\label{eq:3.8}
v_\epsilon(x,t)\to\infty\quad\text{as}\quad |x|\to\infty\quad\text{for $t\in[0,T(\phi))$}.
\end{equation}
Furthermore, $v_\epsilon$ satisfies
\begin{equation}
\label{eq:3.9}
\partial_t v_\epsilon=\Delta v_\epsilon+g_F(v_\epsilon)|\nabla v_\epsilon|^2\quad
\text{in}\quad{\mathbb R}^n\times(0,T(\phi)),
\end{equation}
 {where $g_F$ is as in \eqref{eq:1.4}.} 
Let $\lambda\in(0,1)$, and define
\[
V_{\lambda,\epsilon}(x,t)
:=\inf\Bigl\{(1-\lambda)v_\epsilon(x_0,t)+\lambda v_\epsilon(x_1,t)
:\, x=(1-\lambda)x_0+\lambda x_1,\; x_0,x_1\in{\mathbb R}^n\Bigr\}
\]
for $(x,t)\in{\mathbb R}^n\times[0,T(\phi))$. 
This, together with \eqref{eq:3.8}, implies that
\begin{equation}
\label{eq:3.10}
V_{\lambda,\epsilon}\in C({\mathbb R}^n\times[0,T(\phi))),\qquad 
V_{\lambda,\epsilon}(x,t)\le v_\epsilon(x,t)
\quad\text{for }(x,t)\in{\mathbb R}^n\times[0,T(\phi)).
\end{equation}
Since $g_F$ is convex, 
by Proposition~\ref{Proposition:3.1}, together with \eqref{eq:3.8} and \eqref{eq:3.9}, 
we see that $V_{\lambda,\epsilon}$ is a viscosity supersolution of \eqref{eq:3.9}.
Define
\[
U_{\lambda,\epsilon}(x,t):=F^{-1}\!\left(V_{\lambda,\epsilon}(x,t)\right),
\quad (x,t)\in{\mathbb R}^n\times[0,T(\phi)). 
\]
Then $U_{\lambda,\epsilon}\in C({\mathbb R}^n\times[0,T(\phi)))$, 
and $U_{\lambda,\epsilon}$ is a viscosity supersolution of the heat equation in
${\mathbb R}^n\times(0,T(\phi))$. 
Furthermore, by \eqref{eq:3.10}, we have
\begin{equation}
\label{eq:3.11}
U_{\lambda,\epsilon}(x,t)\le u_\epsilon(x,t)
\quad\text{for }(x,t)\in{\mathbb R}^n\times[0,T(\phi)).
\end{equation}

Let $T\in(0,T(\phi))$.
By a chain of parabolic Harnack inequalities (see, e.g., \cite{A}*{Theorem~E} and \cite{IM}*{Proposition~2.2}), 
for any $\delta\in(0,T/2)$, 
there exists $A>0$ and $a>0$ such that 
\begin{equation}
\label{eq:3.12}
u(x,t)\le ae^{A|x|^2},\quad (x,t)\in{\mathbb R}^n\times[\delta,T-\delta].
\end{equation}
Define a function $w$ on ${\mathbb R}^n\times[\delta,T-\delta)$ by
\[
w(x,t):=\bigl[e^{(t-\delta)\Delta}U_{\lambda,\epsilon}(\delta)\bigr](x)
- U_{\lambda,\epsilon}(x,t),
\quad (x,t)\in{\mathbb R}^n\times[\delta,T-\delta).
\]
Then $w$ is a viscosity subsolution of the heat equation in ${\mathbb R}^n\times(\delta,T-\delta)$ 
such that $w=0$ on ${\mathbb R}^n\times\{\delta\}$.
Furthermore, by \eqref{eq:3.7}, \eqref{eq:3.11}, and \eqref{eq:3.12}, we have
\begin{align*}
w(x,t)
&\le \bigl[e^{(t-\delta)\Delta}U_{\lambda,\epsilon}(\delta)\bigr](x)\\
&\le \bigl[e^{(t-\delta)\Delta}u_\epsilon(\delta)\bigr](x)
 =u_\epsilon(x,t)=u(x,t)+\epsilon(|x|^2+2nt)
 \le C e^{A|x|^2}
\end{align*}
for $(x,t)\in {\mathbb R}^n\times[\delta,T-\delta)$, where $C>0$ is a constant.
Therefore, by Lemma~\ref{Lemma:2.3}, we conclude that
\[
w\le 0 \quad\text{on } {\mathbb R}^n\times[\delta,T-\delta).
\]
By the arbitrariness of $\delta\in(0,T/2)$, this implies that 
\[
 \bigl[e^{t\Delta}U_{\lambda,\epsilon}(0)\bigr](x)\le U_{\lambda,\epsilon}(x,t),
 \quad (x,t)\in{\mathbb R}^n\times(0,T).
\]
Since $T\in(0,T(\phi))$ is arbitrary and 
$U_{\lambda,\epsilon}\in C({\mathbb R}^n\times[0,T(\phi)))$, 
we obtain
\begin{equation}
\label{eq:3.13}
 \bigl[e^{t\Delta}U_{\lambda,\epsilon}(0)\bigr]\bigl((1-\lambda)x_0+\lambda x_1\bigr)
 \le 
 F^{-1}\!\left((1-\lambda)F(u_\epsilon(x_0,t))+\lambda F(u_\epsilon(x_1,t))\right)
\end{equation}
for $x_0,x_1\in{\mathbb R}^n$ and $t\in(0,T(\phi))$. 
Since $\phi$ is $F$-convex on ${\mathbb R}^n$, it follows that
\begin{equation}
\label{eq:3.14}
U_{\lambda,\epsilon}(x,0)=F^{-1}(V_{\lambda,\epsilon}(x,0))\ge F^{-1}(F(\phi(x)))=\phi(x),\quad x\in{\mathbb R}^n.
\end{equation}
Therefore, by \eqref{eq:3.6}, \eqref{eq:3.13}, and \eqref{eq:3.14}, 
and letting $\epsilon\to 0^+$, we obtain 
\[
u\bigl((1-\lambda)x_0+\lambda x_1,t\bigr)
\le 
F^{-1}\!\left((1-\lambda)F(u(x_0,t))+\lambda F(u(x_1,t))\right)
\]
for $x_0,x_1\in{\mathbb R}^n$ and $t\in(0,T(\phi))$. 
Hence, $F$-convexity is preserved  {under} the heat flow.

Next, we assume that $F$-convexity is preserved  {under} the heat flow. 
Since $F\in C^2((0,\infty))$ is strictly increasing on $(0,\infty)$ and $\lim_{r\to\infty}F(r)=\infty$,
for any $L\in(0,\infty)$, 
there exists $r_*\in(L,\infty)$ such that $F'(r_*)>0$. 
Then we can find $r_1,r_2\in (0,\infty)$ with $r_1<r_*<r_2$ such that 
\begin{equation}
\label{eq:3.15}
F'(r)>0\quad\text{for}\quad r\in(r_1,r_2).
\end{equation}
Let $r_0\in(0,r_1)$. Set
\begin{equation}
\label{eq:3.16}
\phi(x):=
\begin{cases}
f_F(\xi) & \text{if } x=(\xi,x_2,\dots,x_n)\in{\mathbb R}^n \text{ and } \xi\ge F(r_0),\\
f_F(-\xi+2F(r_0)) & \text{if } x=(\xi,x_2,\dots,x_n)\in{\mathbb R}^n \text{ and } \xi< F(r_0).
\end{cases}
\end{equation}
Then $\phi$ is $F$-convex in ${\mathbb R}^n$. Moreover, by condition~\eqref{eq:F}
we have $T(\phi)\in(0,\infty]$, and the function
$u(x,t):=[e^{t\Delta}\phi](x)$ is well defined for
$(x,t)\in{\mathbb R}^n\times[0,T(\phi))$.
For any $\lambda\in(0,1)$, define
\begin{align*}
\Phi_\lambda(x,y,t)
:=&
F\bigl(u((1-\lambda)x+\lambda y,t)\bigr)
-(1-\lambda)F(u(x,t))-\lambda F(u(y,t))
\end{align*}
for $x,y\in{\mathbb R}^n$ and $t\in[0,T(\phi))$.
Since $\phi$ is $F$-convex and $F$-convexity is preserved  {under} the heat flow, 
it follows that 
\begin{align*}
 & \Phi_\lambda(x,y,t)\le 0
 \quad\text{for } x,y\in{\mathbb R}^n,\ t\in[0,T(\phi)),\\
 & \Phi_\lambda(x,y,0)=0
 \quad\text{for } x_1,y_1\leq F(r_0)\,\text{ or }\, x_1,y_1\geq F(r_0).
\end{align*}
Hence,
\[
(\partial_t\Phi_\lambda)(x,y,0)\le 0
\quad\text{for }  x_1,y_1< F(r_0)\,\text{ or }\, x_1,y_1>F(r_0).
\]
Using $\partial_t u=\Delta u$ and the chain rule, by \eqref{eq:3.16}, we obtain
\begin{align}
0
&\ge \partial_t\Phi_\lambda(x,y,0) \nonumber\\
&=F'(u((1-\lambda)x+\lambda y,0))\,\Delta u((1-\lambda)x+\lambda y,0)\nonumber\\
&\quad-(1-\lambda)F'(u(x,0))\,\Delta u(x,0)
-\lambda F'(u(y,0))\,\Delta u(y,0)\label{eq:3.17}
\end{align}
for $x=(x_1,0,\dots,0)$, $y=(y_1,0,\dots,0)\in{\mathbb R}^n$ with $x_1$, $y_1\in(F(r_1),F(r_2))$. 
This implies that
\begin{align}
0
&\ge F'(f_F((1-\lambda)x_1+\lambda y_1))\,f_F''((1-\lambda)x_1+\lambda y_1)\nonumber\\
&\quad-(1-\lambda)F'(f_F(x_1))\,f_F''(x_1)
-\lambda F'(f_F(y_1))\,f_F''(y_1)\label{eq:3.18}
\end{align}
for $x_1$, $y_1\in(F(r_1),F(r_2))$. 
Since
\[
F'(f_F(z))=\frac{1}{f_F'(z)} \quad\text{for } z\in(F(r_1),F(r_2)),
\]
it follows from \eqref{eq:3.15} and \eqref{eq:3.18} that
\begin{equation}\label{eq:3.19}
g_F((1-\lambda)x_1+\lambda y_1)
\le (1-\lambda)g_F(x_1)+\lambda g_F(y_1)
\end{equation}
for all $x_1,y_1\in(F(r_1),F(r_2))$.
Hence, $g_F$ is convex on $(F(r_1),F(r_2))$.

Assume that $F'(r_1)=0$. It follows from \eqref{eq:3.19} that 
\[
\liminf_{z\to F(r_1)^+} (\log f_F')'(z)
   = \liminf_{z\to F(r_1)^+} g_F(z) > -\infty,
\]
which implies that for any $z_0\in(F(r_1),F(r_2))$,
\[
\log f_F'(z_0)
 - \lim_{z\to F(r_1)^+}\log f_F'(z)
 = \lim_{z\to F(r_1)^+}\int_z^{z_0} (\log f_F'(w))'\,\mathrm{d}w
 > -\infty.
\]
Hence, $\log f_F'(z)$ is bounded  {from above as $z\to F(r_1)^+$.}
On the other hand, since
\[
F'(f_F(z))\, f_F'(z)=1
\quad \text{for } z\in(F(r_1),F(r_2)),
\]
and $F'(r_1)=0$, we obtain
\[
\lim_{z\to F(r_1)^+} f_F'(z)
 = \lim_{z\to F(r_1)^+}\frac{1}{F'(f_F(z))} = \infty.
\]
This contradicts the upper boundedness of $\log f_F'(z)$  {as $z\to F(r_1)^+$}. 
Therefore, $F'(r_1)>0$, and we see that $F'>0$ on $(0,L)$ and 
$g_F$ is convex on $(F(0^+),L)$. 
Since $L>0$ is arbitrary, we conclude that 
$F'(r)>0$ for all $r>0$, 
and $g_F$ is convex on $(F(0^+),\infty)$. 
Hence, \eqref{eq:1.4} holds.
This proves assertion~(3), and thus completes the proof of 
Theorem~\ref{Theorem:1.1}.
$\Box$
\vspace{5pt}
\newline
{\bf Proof of Theorem~\ref{Theorem:1.2}.}
Assume that $F$-convexity is preserved  {under}  the heat flow. 
Then, by Theorem~\ref{Theorem:1.1}, we have 
\begin{equation}
\label{eq:3.20}
\text{$f_F'>0$ on $J_F$},
\qquad 
\text{$g_F:=(\log f_F')'$ is convex on $J_F$}.
\end{equation}
Here $J_F=(F(0^+),\infty)$. 
Fix $z'\in J_F$. 
By the definition of $g_F$, we have
\begin{equation}
\label{eq:3.21}
f_F(z)
= f_F(z')
+ f'_F(z')
\int_{z'}^z
\exp\!\left(\int_{z'}^s g_F(\tau)\,\mathrm{d}\tau\right)\mathrm{d}s,
\quad z\in J_F. 
\end{equation}

We claim that 
\begin{equation}
\label{eq:3.22}
\text{$g_F$ is non-increasing on $J_F$}. 
\end{equation}
Assume to the contrary that $g_F(z_1)>g_F(z_0)$ for some $z_0$, $z_1\in J_F$ with $z_0<z_1$. 
By the convexity of $g_F$, there exist $z_2\in[z_1,\infty)$ and $\ell>0$ such that 
\[
g_F(z)\ge \ell(z-z_2),\quad z\in[z_2,\infty).
\]
Applying \eqref{eq:3.21} with $z'=z_2$, we have
\begin{equation*}
\begin{split}
 & f_F(z)-f_F(z_2)\\
& =f'_F(z_2)\int_{z_2}^z
\exp\!\left(\int_{z_2}^s g_F(\tau)\,\mathrm{d}\tau\right)\mathrm{d}s
\ge f'_F(z_2)\int_{z_2}^z
\exp\!\left(\ell\int_{z_2}^s(\tau-z_2)\,\mathrm{d}\tau\right)\mathrm{d}s\\
&= f'_F(z_2)\int_{z_2}^z
\exp\!\left(\frac{\ell}{2}(s-z_2)^2\right)\mathrm{d}s\ge C\exp\!\left(\frac{\ell}{4}(z-z_2)^2\right)
\end{split}
\end{equation*}
for all $z\ge z_2+1$, where $C>0$ is a constant.
Hence,
\[
f_F(z)\ge C\exp\!\left(\frac{\ell}{8}z^2\right)
\quad\text{for all sufficiently large $z$.}
\]
This contradicts condition~\eqref{eq:F'}. 
Therefore, \eqref{eq:3.22} holds. 

Next, we show that 
\begin{equation}
\label{eq:3.24}
g_F(z)=\frac{f_F''(z)}{f_F'(z)}\ge 0,
\quad z\in J_F.
\end{equation}
Assume that $g_F(z_3)<0$ for some $z_3\in J_F$. 
Then it follows from \eqref{eq:3.22} that
$$
g_F(z)\le g_F(z_3)<0,\quad z\in[z_3,\infty).
$$
Applying \eqref{eq:3.21} with $z'=z_3$, we have
$$
f_F(z) 
\le f_F(z_3)
+ C\int_{z_3}^z
\exp\!\left(g_F(z_3)(s-z_3)\right)\,\mathrm{d}s
$$
for all $z\in[z_3,\infty)$, where $C>0$ is a constant. 
This contradicts 
$\lim_{z\to\infty} f_F(z)=\infty$. 
Therefore, \eqref{eq:3.24} holds.
This implies that $f_F$ is convex on $J_F$. 
Then, by Lemma~\ref{Lemma:2.1}, we conclude that convexity is weaker than $F$-convexity.

Next, we prove that $f_F$ is log-concave on $J_F$, that is, 
\begin{equation}
\label{eq:3.25}
f_F(z)\,f_F''(z)-(f_F'(z))^2\le 0,\quad z\in J_F.
\end{equation}
If $f_F''\le 0$ on $J_F$, then \eqref{eq:3.25} obviously holds. 
Hence, it suffices to consider the case
\begin{equation}
\label{eq:3.26}
f_F''(z_*)>0\quad\mbox{for some $z_*\in J_F$}.
\end{equation}
Then $g_F(z_*)>0$, which, together with \eqref{eq:3.22}, implies that
\begin{equation}
\label{eq:3.27}
g_F(z)>0,\quad z\in(F(0^+),z_*].
\end{equation}
On the other hand, since $\lim_{z\to F(0^+)^+}f_F(z)=0$, 
by \eqref{eq:3.21} with $z'=z_*$, we have
$$
0=f_F(z_*)+f'_F(z_*)\int_{z_*}^{F(0^+)}
\exp\!\left(\int_{z_*}^s g_F(\tau)\,\mathrm{d}\tau\right)\,\mathrm{d}s.
$$
Using \eqref{eq:3.21} with $z'=z_*$ again, we obtain
$$
f_F(z)
= f_F'(z_*)\int_{F(0^+)}^z
\exp\!\left(\int_{z_*}^s g_F(\tau)\,\mathrm{d}\tau\right)\,\mathrm{d}s,
\quad z\in(F(0^+), z_*],
$$
and hence
\begin{equation}
\label{eq:3.28}
f_F'(z)
= f_F'(z_*)\exp\!\left(\int_{z_*}^z g_F(\tau)\,\mathrm{d}\tau\right),
\quad z\in(F(0^+), z_*],
\end{equation}

Define 
$$
h_F(z):=f_F(z)-\frac{f_F'(z_*)}{g_F(z)}
\exp\left(\int_{z_*}^z g_F(\tau)\,\mathrm{d}\tau\right),
\quad z\in (F(0^+),z_*].
$$
By \eqref{eq:3.20} and \eqref{eq:3.27}, we have
\begin{equation}
\label{eq:3.29}
h_F(z)\leq f_F(z),\quad z\in(F(0^+), z_*].
\end{equation}
Differentiating $h_F$ and using \eqref{eq:3.22} and \eqref{eq:3.28}, 
we obtain
\begin{equation}
\label{eq:3.30}
\begin{split}
h_F'(z) & =f_F'(z)+\frac{f_F'(z_*)g_F'(z)}{g_F(z)^2}\exp\left(\int_{z_*}^z g_F(\tau)\,\mathrm{d}\tau\right)
-f_F'(z_*)\exp\left(\int_{z_*}^z g_F(\tau)\,\mathrm{d}\tau\right)\\
 & \le f_F'(z)-f_F'(z_*)\exp\left(\int_{z_*}^z g_F(\tau)\,\mathrm{d}\tau\right)=0
\end{split}
\end{equation}
for $z\in (F(0^+),z_*]$.
On the other hand,  {it follows from \eqref{eq:3.29} that}
\begin{equation}
\label{eq:3.31}
\limsup_{z\to F(0^+)^+} h_F(z)
\le \limsup_{z\to F(0^+)^+} f_F(z)=0.
\end{equation}
Therefore, \eqref{eq:3.30} and \eqref{eq:3.31} imply that $h_F\le 0$ on $(F(0^+),z_*]$. 
In particular,
$$
0\ge h_F(z_*)
= f_F(z_*)-\frac{f_F'(z_*)}{g_F(z_*)}
= f_F(z_*)-\frac{(f_F'(z_*))^2}{f_F''(z_*)}.
$$
Hence, by \eqref{eq:3.26}, we obtain \eqref{eq:3.25}. 
This shows that $f_F$ is log-concave on $J_F$. 
Then, by Lemma~\ref{Lemma:2.1}, we conclude that $F$-convexity is weaker than log-convexity.
This completes the proof of Theorem~\ref{Theorem:1.2}.
$\Box$
%%%%%%%%%%%%%%%%%%%%%%
%%%%%%%%%%%%%%%%%%%%%%
\section{Preservation of $F$-convexity without any sign restriction}
%%%%%%%%%%%%%%%%%%%%%%
%%%%%%%%%%%%%%%%%%%%%%
In this section, we apply the arguments developed above to study the preservation of convexity properties under the heat flow without any sign restriction on the solutions.
As in the case of nonnegative solutions, if $\phi$ is continuous and bounded from below on ${\mathbb R}^n$, 
then problem~\eqref{eq:H} admits a classical solution in ${\mathbb R}^n\times(0,T)$ for some $T\in(0,\infty]$ if and only if $\phi$ satisfies~\eqref{eq:1.1}.
Furthermore, under condition~\eqref{eq:1.1}, $e^{t\Delta}\phi$ is the unique classical solution to problem~\eqref{eq:H} 
that is bounded from below by $\inf_{{\mathbb R}^n} \phi$. 
Taking this into account, we define admissible functions on ${\mathbb R}$ and discuss the preservation of $F$-convexity  {under} the heat flow.
\begin{definition}
\label{Definition:4.1}
A function $F:{\mathbb R}\rightarrow{\mathbb R}$ is said to be admissible on ${\mathbb R}$ if 
$F$ is strictly increasing and continuous on ${\mathbb R}$. 
Denote by $f_F$ the inverse function of $F$ on $J_F:=F({\mathbb R})$. 
\end{definition}
 {We extend Definition~\ref{Definition:1.2} to functions bounded from below in a natural way.}
In particular, in this section we say that $F$-convexity is preserved  {under}  the heat flow if, for any $F$-convex function~$\phi$ bounded from below on ${\mathbb R}^n$ 
with $T(\phi)\in(0,\infty]$, the function $e^{t\Delta}\phi$ remains $F$-convex for all $t\in(0,T(\phi))$. 
Then we have the following results.
\begin{theorem}
\label{Theorem:4.1}
Let $F$ be admissible on ${\mathbb R}$. 
Then the assertions of Theorem~{\rm\ref{Theorem:1.1}} remain valid, with~$\mathbb{R}$ in place of $(0,\infty)$. 
\end{theorem}
{\bf Proof.}
The proof is the same as that of Theorem~\ref{Theorem:1.1} and is therefore omitted.
We only note that, in the present setting, the discussion of the preservation of $F$-convexity is restricted to initial data that are bounded from below.
$\Box$
\begin{theorem}
\label{Theorem:4.2}
Let $F\in C^2({\mathbb R})$ be admissible on ${\mathbb R}$ and assume that
\begin{equation}
\label{eq:4.1}
\int_{F({\mathbb R})} e^{-A z^2}\,|f_F(z)|\,\mathrm{d}{z}<\infty
\quad \text{for any } A>0.
\end{equation}
Then $F$-convexity is preserved  {under}  the heat flow if and only if 
there exists a pair $(A,B)\in(0,\infty)\times{\mathbb R}$ such that
$$
F(r)=Ar+B\quad \text{for all } r\in{\mathbb R}.
$$ 
That is, $F$-convexity coincides with the usual convexity.
\end{theorem}
{\bf Proof.}
Let $F\in C^2({\mathbb R})$ be admissible on ${\mathbb R}$ and assume \eqref{eq:4.1}. 
Suppose  {that} $F$-convexity is preserved under  the heat flow. 
By Theorem~\ref{Theorem:4.1}, we see that \eqref{eq:1.4} holds. 
Furthermore, applying the same arguments as in \eqref{eq:3.22} and \eqref{eq:3.24}, 
we find that $g_F$ is non-increasing, nonnegative, and convex on $J_F$. 

Let $z_*:=\inf_{r\in{\mathbb R}}F(r)$. 
Since the domain of $F$ is ${\mathbb R}$, it follows that 
\begin{equation}
\label{eq:4.2}
\lim_{z\to z_*^+}f_F(z)=-\infty.
\end{equation}
Since $g_F\ge 0$ on $J_F$, 
if $z_*>-\infty$, then, by \eqref{eq:3.21}, we see that $f_F$ is bounded from below, which contradicts \eqref{eq:4.2}. 
Hence, $z_*=-\infty$, that is, $J_F:=F({\mathbb R})={\mathbb R}$ and 
\begin{equation}
\label{eq:4.3}
\lim_{z\to -\infty}f_F(z)=-\infty.
\end{equation}

We prove by contradiction that $g_F$ is constant on ${\mathbb R}$.
Assume that $g_F$ is not constant on ${\mathbb R}$.
Since $g_F$ is non-increasing, there exist $z_1<z_0$ such that
\[
g_F(z_1)>g_F(z_0).
\]
By the convexity and the nonnegativity of $g_F$, we have
$$
g_F(z)\ge g_F(z_1)+\ell(z-z_1)\geq \ell(z-z_1)\quad\text{ for } z\in(-\infty, z_1],
$$
with
$$
\ell=\frac{g_F(z_0)-g_F(z_1)}{z_0-z_1}<0\,.
$$
This, together with \eqref{eq:3.20} and \eqref{eq:3.21}, implies that 
\begin{equation*}
\begin{split}
 f_F(z)-f_F(z_1)
 & = f'_F(z_1)\int_{z_1}^z
 \exp\left(\int_{z_1}^s g_F(\tau)\,\mathrm{d}\tau\right)\,\mathrm{d}s \\
 & \ge f'_F(z_1)\int_{z_1}^z
 \exp\left(\ell\int_{z_1}^s(\tau-z_1)\,\mathrm{d}\tau\right)\,\mathrm{d}s \\
 & = f'_F(z_1)\int_{z_1}^z 
 \exp\left(\frac{\ell}{2}(s-z_1)^2\right)\,\mathrm{d}s
\end{split}
\end{equation*}
for $z\in(-\infty,z_1]$. 
Since $\ell<0$, the integral on the right-hand side is bounded from below uniformly in $z\le z_1$. 
Hence $f_F$ is bounded from below on $(-\infty,z_1]$, which contradicts \eqref{eq:4.3}.

Next, we prove that $g_F=0$ on ${\mathbb R}$. 
Since $g_F$ is constant on ${\mathbb R}$ and nonnegative, there exists $c\ge0$ such that
\[
g_F(z)=c \qquad \text{for } z\in{\mathbb R}.
\]
Assume that $c>0$. Then, by \eqref{eq:3.21}, for any fixed $z_2\in{\mathbb R}$ we have
\[
f_F(z)
= f_F(z_2)+f'_F(z_2)\int_{z_2}^z e^{c(s-z_2)}\,\mathrm{d}s
= f_F(z_2)+\frac{f'_F(z_2)}{c}\bigl(e^{c(z-z_2)}-1\bigr),
\quad z\in{\mathbb R}.
\]
This shows that $f_F$ is bounded from below, which contradicts \eqref{eq:4.3}.
Hence, $c=0$, and therefore $g_F\equiv0$ on ${\mathbb R}$.
Then, 
for any fixed $z'\in{\mathbb R}$, 
it follows from \eqref{eq:3.21} that
$$
f_F(z)=f_F(z')+f'_F(z')(z-z'),
\quad z\in(-\infty,\infty),
$$
that is, $f_F$ is affine on ${\mathbb R}$, and hence $F$ is affine as well.
Therefore, $F$-convexity coincides with the usual convexity.
This completes the proof of Theorem~\ref{Theorem:4.2}.
$\Box$
%%%%%%%%%%%%%%%%%%%%%%
%%%%%%%%%%%%%%%%%%%%%%
\section{Preservation of $F$-convexity under the Dirichlet heat flow}
%%%%%%%%%%%%%%%%%%%%%%
%%%%%%%%%%%%%%%%%%%%%%
Consider the Cauchy--Dirichlet problem for the heat equation
\begin{equation}
\tag{DH}
\label{eq:DH}
\left\{
\begin{array}{ll}
\partial_t u=\Delta u\quad & \text{in}\quad\Omega\times(0,T),\vspace{3pt}\\
u=\ell \quad & \text{on}\quad\partial\Omega\times(0,T)\quad\text{if}\quad\partial\Omega\not=\emptyset,\vspace{3pt}\\
u=\phi\quad & \text{on}\quad\Omega\times\{0\}, 
\end{array}
\right.
\end{equation}
where $\Omega$ is a convex domain in ${\mathbb R}^n$ with $\partial\Omega\not=\emptyset$, 
$\ell\in{\mathbb R}$, and $\phi$ is a bounded continuous function on $\Omega$. 
Then problem~\eqref{eq:DH} admits a unique classical global-in-time solution, 
which we denote by $e^{t\Delta_{\mathrm{DH}}}\phi$ in this section. 
 {Using} the arguments in \cite{IST05}, 
we study the convexity properties of solutions to problem~\eqref{eq:DH}.
\begin{definition}
\label{Definition:5.1}
Let $-\infty\le a<\ell<\infty$, and set $I:=(a,\ell]$ or $I:=[a,\ell]$. 
A function $F:I\rightarrow{\mathbb R}\cup\{\pm\infty\}$ is said to be admissible on $I$ if 
$F$ is strictly increasing in $I$, continuous on $(a,\ell)$, $F(a)=\lim_{r\to a^+}F(r)$, and $F(\ell)=\infty$.
\end{definition}

\begin{definition}
\label{Definition:5.2}
Let $-\infty\le a<\ell<\infty$, and set $I:=(a,\ell]$ or $I:=[a,\ell]$. 
Let $F$ be admissible on~$I$,  {and let $\Omega$ be a convex domain in ${\mathbb R}^n$.} 
Define
$$
\mathcal{A}_\Omega(I)
:=\{f:\Omega\to{\mathbb R}\mid f(\Omega)\subset I\}.
$$
\begin{itemize}
  \item[{\rm (1)}] 
  Given $f\in\mathcal{A}_\Omega(I)$,
  we say that $f$ is $F$-convex in $\Omega$ if 
  $$
  F(f((1-\lambda)x+\lambda y))
  \le(1-\lambda)F(f(x))+\lambda F(f(y))
  $$
  for all $x$, $y\in {\Omega}$ with 
  $(F(f(x)),F(f(y)))\not=(\pm\infty,\mp\infty)$ and all $\lambda\in(0,1)$. 
  We denote by $\mathcal{C}_\Omega[F]$ the set of $F$-convex functions on $ {\Omega}$.
  For any admissible functions $F_1$ and $F_2$, if ${\mathcal C}_\Omega[F_1]\subset {\mathcal C}_\Omega[F_2]$, 
  we say that $F_1$-convexity is stronger than $F_2$-convexity, 
  or equivalently, that $F_2$-convexity  is weaker than $F_1$-convexity. 
  \item[{\rm (2)}] 
  We say that $F$-convexity  is preserved  {under}  the Dirichlet heat flow~\eqref{eq:DH} if, 
  for any bounded $F$-convex function~$\phi$ on $\Omega$, 
  the solution $e^{t\Delta_{\mathrm{DH}}}\phi$ is $F$-convex for all $t>0$. 
\end{itemize}
\end{definition}
We modify the arguments in \cite{IST05} to introduce the notion of \emph{hot-convexity}. 
\begin{definition}
\label{Definition:5.3}
Let
$$
h(z):=\left(e^{{\Delta}_{{\mathbb R}}}{\bf 1}_{[0,\infty)}\right)(z)
=(4\pi)^{-\frac{1}{2}}\int_0^\infty e^{-\frac{|z-w|^2}{4}}\,\mathrm{d}w,
\quad z\in{\mathbb R}.
$$
Then the function $h$ is smooth on ${\mathbb R}$, 
$\lim_{z\to-\infty}h(z)=0$, $\lim_{z\to\infty}h(z)=1$, and $h'>0$ on ${\mathbb R}$.
We denote by $H$ the inverse function of $h$. 
For any $a\in(0,\infty]$, we define an admissible function $H_a$ on $[0,a]$ by  
\begin{equation*}
H_a(r):=
\begin{cases}
H(r/a), & r\in(0,a),\ a<\infty,\\[3pt]
\log r, & r\in(0,a),\ a=\infty,\\[3pt]
\infty, & r=a,\ a<\infty.
\end{cases}
\end{equation*}
We call $H_a$-convexity \emph{hot-convexity}.
\end{definition}
We then obtain the following results.
\begin{theorem}
\label{Theorem:5.1}
Let $\Omega$ be a convex domain in ${\mathbb R}^n$ with $\partial\Omega\not=\emptyset$. 
Let $F$ be admissible on $I$, where $I:=(a,\ell]$ or $I:=[a,\ell]$ with $-\infty\le a<\ell<\infty$.
\begin{itemize}
\item[{\rm (1)}] 
  $H_{\ell-a}$-convexity is the strongest $F$-convexity preserved  {under}  the Dirichlet heat flow~\eqref{eq:DH}.
\item[{\rm (2)}] 
  Define an admissible function $\Phi_{\ell-a}$ on $I$ by 
  $$
  \Phi_{\ell-a}(r):=
  \begin{cases}
  -\log(\ell-r), & r\in[a,\ell),\\[5pt]
  \infty, & r=\ell.
  \end{cases}
  $$
  Then $\Phi_{\ell-a}$-convexity is the weakest $F$-convexity preserved  {under}  the Dirichlet heat flow~\eqref{eq:DH} for $n\geq 2$; 
  when $n=1$, the same is true under the $C^2$-regularity assumption on $F$ 
  (moreover, quasi-convexity is preserved as well).
\item[{\rm (3)}] 
  When $n\geq 2$, if the initial datum is not $\Phi_{\ell-a}$-convex, 
  even quasi-convexity can be immediately destroyed  {under} the Dirichlet heat flow~\eqref{eq:DH}, 
  thus destroying all nontrivial convexity properties.
\end{itemize}
\end{theorem}
\begin{theorem}
\label{Theorem:5.2}
Let $\Omega$ be a convex domain in ${\mathbb R}^n$ with $\partial\Omega\not=\emptyset$. 
Let $F$ be admissible on $I$, where $I:=(a,\ell]$ or $I:=[a,\ell]$ with $-\infty\le a<\ell<\infty$, and assume that $F\in C^2(\mathrm{int}\,I)$. 
Then $F$-convexity is preserved  {under}  the Dirichlet heat flow~\eqref{eq:DH} if and only if
$$ 
\lim_{r\to \ell^-}F(r)=\infty,\quad F'>0 \ \text{in } \mathrm{int}\,I,
\quad\text{and}\quad (\log f_F')' \ \text{is convex on } F(\mathrm{int}\,I).
$$
\end{theorem}
{\bf Proofs of Theorems~\ref{Theorem:5.1} and \ref{Theorem:5.2}.}
We first consider the case $I=(a,\ell]$. 
For any solution $u$ to problem~\eqref{eq:DH}, define
$$
\tilde{u}(x,t):=\ell-u(x,t),\quad (x,t)\in\overline{\Omega}\times[0,T).
$$
Then $\tilde{u}$ satisfies
\begin{equation*}
\left\{
\begin{array}{ll}
\partial_t \tilde{u}=\Delta \tilde{u}\quad & \text{in}\quad\Omega\times(0,T),\vspace{3pt}\\
\tilde{u}=0 \quad & \text{on}\quad\partial\Omega\times(0,T)\quad\text{if}\quad\partial\Omega\not=\emptyset,\vspace{3pt}\\
\tilde{u}=\ell-\phi\quad & \text{on}\quad\Omega\times\{0\}.
\end{array}
\right.
\end{equation*}
By applying the results on the preservation of $F$-concavity under the Dirichlet heat flow 
(see \cite{IST05}*{Theorems~1.5 and~1.6}), 
we obtain Theorems~\ref{Theorem:5.1} and \ref{Theorem:5.2} in this case.

Next, we consider the case $I=[a,\ell]$. 
Let $\phi$ be $F$-convex in $\Omega$. 
For any $\epsilon\in(0,\ell-a)$, define
$$
\phi_\epsilon(x):=\max\{\phi(x),a+\epsilon\},\quad x\in\Omega.
$$
Let $x_0$, $x_1\in\Omega$ and $\lambda\in(0,1)$. 
If $(F(\phi(x_0)),F(\phi(x_1)))\not=(\pm\infty,\mp\infty)$, then
$$
F(\phi((1-\lambda)x_0+\lambda x_1))
\le(1-\lambda)F(\phi(x_0))+\lambda F(\phi(x_1)).
$$
If $\phi((1-\lambda)x_0+\lambda x_1)\ge a+\epsilon$, then
\begin{align*}
 & F(\phi_\epsilon((1-\lambda)x_0+\lambda x_1))
 =F(\phi((1-\lambda)x_0+\lambda x_1))\\
 & \le(1-\lambda)F(\phi(x_0))+\lambda F(\phi(x_1))
 \le(1-\lambda)F(\phi_\epsilon(x_0))+\lambda F(\phi_\epsilon(x_1)).
\end{align*}
If $\phi((1-\lambda)x_0+\lambda x_1)<a+\epsilon$, then
\begin{align*}
 & F(\phi_\epsilon((1-\lambda)x_0+\lambda x_1))
 =F(a+\epsilon)\\
 & =(1-\lambda)F(a+\epsilon)+\lambda F(a+\epsilon)
 \le(1-\lambda)F(\phi_\epsilon(x_0))+\lambda F(\phi_\epsilon(x_1)).
\end{align*}
If either $F(\phi(x_0))=\infty$ or $F(\phi(x_1))=\infty$, then the inequality
$$
F(\phi_\epsilon((1-\lambda)x_0+\lambda x_1))
\le(1-\lambda)F(\phi_\epsilon(x_0))+\lambda F(\phi_\epsilon(x_1))
$$
holds trivially. 
These considerations show that $\phi_\epsilon$ is $F$-convex in $\Omega$.
Then,  applying the arguments in the case $I=(a,\ell]$ and letting $\epsilon\to 0^+$, 
we obtain the desired results. 
Hence, Theorems~\ref{Theorem:5.1} and \ref{Theorem:5.2} hold.
$\Box$
\medskip

\noindent
{\bf Acknowledgements.}\\
K.I. was supported in part by JSPS KAKENHI Grant Number 25H00591.
T.P. was partially supported by (MUR) PRIN 2022 ``Partial differential equations and related geometric-functional inequalities" grant number 20229M52AS 
and the ``INdAM - GNAMPA Project", CUP E5324001950001.
P.S. was partially supported by European Union
 -- Next Generation EU -- through the project ``Geometric-Analytic Methods for PDEs
and Applications (GAMPA)'', within the PRIN 2022 program, and by INdAM - GNAMPA.
\medskip

\noindent
{\bf Data availability.}\\
No data was generated or analyzed as part of the writing of this paper.
\medskip

\noindent
{\bf  Conflict of interest.}\\
The authors state no conflict of interest.
%%%%%%%%%%%%%%%%%%%%%%%%%%%%%%%%%%%%%
%%%%%%%%%%%    references    %%%%%%%%%%%%%%%%%%
%%%%%%%%%%%%%%%%%%%%%%%%%%%%%%%%%%%%%

 \end{document}